\documentclass[11pt]{article}

\usepackage{amssymb,amsmath,amsthm}
\usepackage{amscd}
\usepackage{longtable}
\usepackage{array}
\usepackage[all]{xy}
\usepackage{a4}
\usepackage{mathptmx}
\usepackage[latin1]{inputenc}
\usepackage{hyperref}
\usepackage{float}
\usepackage[usenames,dvipsnames]{color}
\usepackage{comment}

\usepackage{adjustbox}

\usepackage{graphicx}
\usepackage{epstopdf}

\newtheorem{lema}{Lemma}[section]

\newtheorem{prop}{Proposition}[section]

\newtheorem{rem}{Remark}[section]
\newtheorem{conj}{Conjecture}[section]

\newcommand{\Q}{\mathbb Q}

\newcommand{\Qbar}{{\overline{\mathbb Q}}}
\newcommand{\GalQ}{\Gal(\Qbar/\Q)}

\newcommand{\Z}{\mathbb Z}

\newcommand{\F}{\mathbb F}

\newcommand{\C}{\mathbb C}
\newcommand{\GL}{\mathrm{GL}}

\newcommand{\End}{\operatorname{End}}
\newcommand{\Frob}{\operatorname{Frob}}

\newcommand{\Jac}{\operatorname{Jac}}

\newcommand{\Gal}{\operatorname{Gal}}

\newcommand{\p}{\mathfrak{p}}

\newcommand{\cO}{\mathcal {O}}

\newfont{\gotip}{eufm10 at 12pt}

\newcommand{\gm}{\mbox{\gotip m}}

\newcommand{\gp}{\p}

\newcommand{\on}[1]{\operatorname{#1}}

\newcommand{\nn}{\color{black}}

\newcommand{\Tr}{\operatorname{Tr}}

\begin{document}

\title{The Sato-Tate conjecture for a Picard curve\\
with Complex Multiplication}
\author{Joan-C. Lario and Anna Somoza \\ (with an appendix by Francesc Fité)}
\date{\today}
\maketitle

\abstract{Let $C/\Q$ be the genus $3$ Picard 
curve given by the affine model $y^3=x^4-x$. In this paper we compute its Sato-Tate group, show the
generalized Sato-Tate conjecture for $C$, and compute the statistical moments for the limiting distribution
of the normalized local factors of $C$.}

\section{Introduction}
Serre \cite{Ser12} provides a vast generalization of the Sato-Tate conjecture, which is known to be true for varieties with complex multiplication \cite{Joh13}. As a down-to-earth example, in this paper
we consider the Picard curve defined over $\Q$ 
given by the affine model
$$
C \colon y^3 = x^4 -x \,.
$$ 
One easily checks that $[0:1:0]$ is the unique point of 
$C$ at infinite, and that
$C$ has good reduction at all primes different from 
$3$. The Jacobian variety of $C$ is absolutely simple and it 
has complex multiplication by the cyclotomic field $K=\Q(\zeta)$
where $\zeta$ is a primitive $9$th root of unity.

With the help of Sage, we compile information 
on the number of points that the
reduction of $C$ has 
over finite fields of small characteristic:

\begin{table}[h]
$$
\begin{array}{ccccc}
\hline
\;\;p\;\; & |C(\F_p)| & |C(\F_{p^2})| & |C(\F_{p^3})| \\
\hline
2 & 3 & 5 & 9  \\
5 & 6 & 26 & 126  \\
7 & 8 & 50 & 365  \\
11 & 12 & 122 & 1332   \\
13 & 14 & 170 & 2003  \\
17 & 18 & 392 & 4914  \\
19 & 14 & 302 & 6935  \\
\hline
\end{array}
$$
\caption{Number of points $|C(\F_{p^i})|$.}
\end{table}

For every primer $p$ of good reduction, we consider the local zeta function
$$\zeta (C/\F_p ; s) = \exp\left(\sum_{k\geq 1} |C(\F_{p^k})| \, \frac{p^{-ks}}{k}\right)\, .$$

It follows from Weil's conjectures \cite{We49} that the zeta function is a rational function of $T = p^{-s}$. That is,
\begin{equation*}
\zeta (C/\F_p ; T) = \exp\left(\sum_{k\geq 1} |C(\F_{p^k})| 	\,
\frac{T^k}{k}\right) = \frac{L_p(C,T)}{(1-T)(1-pT)}
\end{equation*}
where the so-called \emph{local factor} of $C$ at $p$ 
$$
L_p(C,T) = \sum_{i=0}^6 b_i T^i = \prod_{i=1}^6 (1-\alpha_i T)
$$
is a polynomial of degree $6$ with integral coefficients and the complex numbers $\alpha_i$ satisfy $\vert\alpha_i\vert = \sqrt{p}$. In particular, it is determined by the three numbers $|C(\F_p)|$, $|C(\F_{p^2})|$, $|C(\F_{p^3})|$ according to:

$$
\begin{array}{c@{\,=\,}l}
b_0 & 1 \\[1pt]
b_1 & |C(\F_p)| - (p+1) \\[1pt]
b_2 & (|C(\F_{p^2})| - (p^2+1)+b_1^2)/2\\[1pt]
b_3 & (|C(\F_{p^3})| - (p^3+1)-b_1^3+3b_2b_1)/3 \\[1pt]
b_4 & p b_2 \\[1pt]
b_5 & p^2 b_1\\[1pt]
b_6 & p^3\,.
\end{array}
$$
For all $m\geq 1$ it holds
$$
| C(\F_{p^m}) | = 1+p^m  - \sum_{i=1}^6 \alpha_i^m \,.
$$
The local factors for small good primes are:
\begin{table}[h]
$$
\begin{array}{cl}
\hline
\;\;p\;\;  & L_p(C,T) \\
\hline
 2 & (1+2T^2)(1-2T^2+4T^4) \\
 5 &  (1+5T^2)(1-5T^2+25T^4) \\
 7 &  1+7T^3+343T^6 \\
 11 &  (1+11T^2)(1-11T^2+121T^4) \\
 13 &  1-65T^3+2197T^6 \\
 17 &  (1+17T^2)^3 \\
 19 & 1-6T-12T^2+169T^3-228T^4-2166T^5+6859T^6 \\
\hline
\end{array}
$$
\caption{Local factors $L_p(C,T)$.}
\end{table}

\vskip 0.2truecm 

Even if for every such prime $p$ all terms of the sequence 
$$|C(\F_p)|, |C(\F_{p^2})|, |C(\F_{p^3})|, 
\dots,  |C(\F_{p^m})|, \, \dots 
\quad (m\geq 1)$$
are determined by the first three, the obtention of these
first three can be a hard computational task as soon as the prime $p$ gets large. However, the presence of complex multiplication
enables the fast computation of the local factors
$L_p(C,T)$ (see Section \ref{numeric}).

For future use, we introduce some
notation. The ring of integers of $K$ will be denoted by ${\mathcal O}=\Z[\zeta]$, and
the unit group ${\mathcal O}^* \simeq \Z/18\Z \times \Z \times \Z$ has generators $\epsilon_0=-\zeta^2$, 
$\epsilon_1=  \zeta^4 - \zeta^3 + \zeta$, 
$\epsilon_2=  \zeta^5 + \zeta^2 - \zeta$. 
Let $\sigma_i$ denote the automorphism of 
$\Gal(K/\Q)$ determined by $\sigma_i(\zeta)= \zeta^i$; one has that $\sigma_2$ generates the Galois group $\Gal(K/\Q)\simeq (\Z/9\Z)^*$.  The unique ramified prime in $K/\Q$ is\break$3\cO = (1+\zeta+\zeta^4)^6$.

Since the Jacobian variety $\operatorname{Jac}(C)$ has complex multiplication, the work of Shimura and Taniyama \cite{ST} ensures the existence of an ideal $\gm$ of the ring of integers $\cO$
and a Grössencharakter 
$\psi \colon I_K(\gm) \to \C^\ast$, where $I_K(\gm)$ stands for the group of fractional ideals coprime with $\gm$, 
$$
\psi(\alpha \cO ) = \prod_{\sigma \in \Phi^*} {\,}^\sigma \alpha
\quad \text{if $\alpha \equiv 1 \,(\!\bmod^{\ast} \,\gm)$,
} 
$$
such that
$L(\psi,s) = L(C,s) $. 
The infinite type $\Phi^*$ 
is the reflex of the CM-type $\Phi$ of 
$\operatorname{Jac}(C)$.
Up to a finite number of Euler factors, one has
$$
 L(\psi,s) = 
 \prod_{\gp} \left( 1 - 
 \psi(\gp)
 \operatorname{N}(\gp)^{-s} \right)^{-1}
\text{ and \ \  }
 L(C,s) = 
 \prod_{p}   L_p(C,p^{-s})^{-1} \,.
$$
Hence, the local factor $L_p(C,T)$
can be obtained from the (monic) irreducible polynomial of 
$\psi(\gp)$ over $\Q$ according to
$$
L_p(C,T)=T^6 \operatorname{Irr}(\psi(\gp), 1/T^f;\Q)^{6/(fd)}\,,
$$
where $f$ is the residual class degree of $p$ in $K$, and 
$d=[\Q(\psi(\gp))\colon \Q]$.

\begin{lema}  
There exists a Gr\"ossencharakter $\psi\colon I_K(\gm)\to \C^*$ of conductor
$\gm = (1+\zeta+\zeta^4)^4$ and infinite type 
$\Phi^*= 
\{ \sigma_1,\sigma_5,\sigma_7 \}= 
\{ \sigma_2^0,\sigma_2^4,\sigma_2^5 \}$.
\end{lema}

\begin{proof}
The following holds
$$\epsilon_0^{18} \equiv 1 \pmod \gm \,,\qquad\epsilon_1^9 \equiv 1 \pmod \gm \,,\qquad \epsilon_2^3 
\equiv 1 \pmod \gm \,.
$$
Moreover, one readily checks that
$
\epsilon_0^a \epsilon_1^b \epsilon_2^c \equiv 1 \,(\!\bmod \gm)$
if and only if
$$
\begin{array}{l@{\,\equiv\,}l}
(a,b,c) \quad \, & \quad (0,0,0), \quad (2,1,2), \quad (4,2,1),  \\[3pt]
        & \quad (6,3,0), \quad (8,4,2), \quad (10,5,1), \\[3pt]
        & \quad (12,6,0), \quad (14,7,2), \quad (16,8,1), 
\end{array}
$$
mod $(18,9,3)$, respectively.
Now an easy computation case-by-case shows that if 
$\epsilon_0^a \epsilon_1^b \epsilon_2^c \equiv 1 \pmod \gm$, then 
$$
\prod_{\sigma \in \Phi ^*} {\,}^\sigma (\epsilon_0^a \epsilon_1^b \epsilon_2^c) = 1
\,.
$$

By using that $K$ has class number one, we define $\psi(\gp)$ over prime ideals $\gp$ of 
$\cO$ coprime with $\gm$ as follows. First we find a generator of 
$\gp=(\alpha)$, and then search for
$$
\epsilon_0^a \epsilon_1^b \epsilon_2^c \alpha \equiv 1 \pmod \gm
$$
with $0\leq a < 18$, $0\leq b < 9$, and $0\leq c < 3$.
The existence of such triple $(a,b,c)$
 is 
guaranteed by the fact that $(\alpha,\gm)=1$ and 
the classes of the 486 possible products
$\epsilon_0^a \epsilon_1^b \epsilon_2^c$ exhaust
the all the elements in $(\cO/\gm)^*$. It follows that
$$
\psi(\gp)=  \prod_{\sigma \in \Phi^*} 
{\,}^\sigma (\epsilon_0^a \epsilon_1^b \epsilon_2^c\alpha)  
$$
is well-defined. Finally, one extends $\psi$ over all ideals prime to $\gm$ multiplicatively. An argument along the same lines shows the non existence of a 
Gr\"ossen\-charakter of $K$ of modulus
$(1+\zeta+\zeta^4)^i$ for $i<4$. Thus, $\psi$ has conductor $\gm$. 
\end{proof}

\begin{table}[h]
\centering
\begin{adjustbox}{max width = .9\textwidth}
$
\begin{array}{rrr}
\hline
p& \psi(\gp) 
& L_p(C,T)=T^6\operatorname{Irr}(\psi(\gp), 1/T^f;\Q)^{6/(fd)} \\
\hline
5 & -5 &  (1+5T^2)(1-5T^2+25T^4)  \\
7  & -3 \zeta^{3} - 2 &   1+7T^3+343T^6  \\
11  & -1331 & (1+11T^2)(1-11T^2+121T^4)   \\
13   & 3 \zeta^{3} + 4 &  1-65T^3+2197T^6  \\
17  & 2 \zeta^{5} + \zeta^{4} + \zeta^{2} - \zeta + 1 & 
{\left(1 +17 T^{2} \right)}^{3}\\
19 & -\zeta^{4} - 2 \zeta^{3} - 2 \zeta  &  1-6T-12T^2+169T^3-228T^4-2166T^5+6859T^6  \\
23 &  -23   &  {\left(1+23 T^{2} \right)} 
{\left(1 - 23 \, T^{2} + 529T^{4} \right)} \\
29  & -29 &  {\left(1+29T^{2}\right)}    
{\left( 1 - 29 \, T^{2} + 841 T^{4}\right)} \\
31  &  6 \zeta^{3} + 1 &  1  + 124 \, T^{3} + 29791  T^{6} \\
37  & \zeta^{5} - \zeta^{3} + 2 \zeta^{2} - 1 &     
1 - 6 \, T + 42 \, T^{2}
- 47 \, T^{3} + 1554 \, T^{4} - 8214 \, T^{5} + 50653 T^{6} \\
\hline
\end{array}
$
\end{adjustbox}
\medskip
\caption{Values of the Gr\"ossencharakter $\psi$.}
\end{table}

\begin{prop}\label{prop1.1} 
Let $\psi$  be the above Gr\"ossencharakter. Then, one has $L(C,s)=L(\psi,s)$.
\end{prop}

\begin{proof}
For every prime 
$\gp$ in $I_K(\gm)$, let $\F_{\gp} = \cO/\gp$
be the residue field of $\gp$ and consider the
character $\chi_{\gp} \colon \F_{\gp}^* \to K^*$ such that
$$\chi_{\gp}(x) \equiv x^{(\on{N}(\gp)-1)/9} 
\,(\!\bmod \, \gp)\,,$$
that we extend by $\chi_{\gp}(0)=0$. By Hasse \cite{Has}, the Jacobi sum
$$
\on{J}(\gp) := 
 - \sum_{x\in \F_{\gp}} \chi_{\gp}^3(x) \chi_{\gp}(1-x)  $$
is uniquely determined by the three properties:
\begin{itemize}
\item[(i)] $|J(\gp)| = \sqrt{\on{N}(\gp)}\,$;
\item[(ii)] $J(\gp) \equiv 1 \,(\!\bmod \,\gm)$\,;
\item[(iii)] $J(\gp) \,\cO = (\gp \cdot \gp^{\sigma_2^4} \cdot \gp^{\sigma_2^5})$\,.
\end{itemize}
One the one hand, it is easy to check that $\psi(\gp)$ satisfies (i), (ii), and (iii). On the other hand, 
Holzapfel and Nicolae \cite{Hol-Nic03} show that for a primer power $q$ such that $q\not\equiv 1 \,(\!\bmod \,9)$ one has $|C(\F_q)| = q +1$, while for $q\equiv 1 \,(\!\bmod \, 9)$ it follows
$$
|C(\F_{\gp})| = \on{N}(\gp) +1 - \on{Tr}_{K/\Q}(J(\gp)) \,,
$$ 
where $\gp$ is any prime ideal of the factorization of $q\mathcal{O}$.
The claim follows.
\end{proof}

\begin{rem}
The proof of the last equalities takes 4 pages in the referenced article \cite{Hol-Nic03}. We are grateful to Francesc Fité for 
a more concise proof included in the appendix of the present paper.
\end{rem}

The Gr\"ossencharakter $\psi$ satisfies ${}^\sigma \psi(\gp) = \psi({}^\sigma \gp ) $ for every prime ideal $\gp$ and $\sigma \in \Gal(K/\Q)$. The $L$-function of the curve $C$ over $K$ satisfies 
$$L(C_K,s) = \prod_{\sigma\in \Gal(K/\Q)} L({\,}^\sigma\psi, s) = L(C,s)^6\,.$$ 
The CM-type of $\on{Jac}(C)$ is 
$\Phi=\{ \sigma_2, \sigma_4, \sigma_8 \}$, i.e.
the reflex of $\Phi^*$.

\section{The Sato-Tate group $\operatorname{ST}(C)$}

For every prime $p\neq 3$, let us normalize
the polynomials
$$
L_p^{\operatorname{ST}}(C,T) = L_p\left(C,\frac{T}{\sqrt{p}}\right)
$$
and call them \emph{normalized local factors}
of $C$. Since they are monic, palindromic
with real coefficients, roots lying in the unit circle and Galois stable, one can think of them as the characteristic polynomials of (conjugacy classes of) matrices in the unitary symplectic group
$$\operatorname{USp}(6,\C)
=\{ M \in \GL(6,\C)\colon M^{-1}= J^{-1} M^{t} J = M^*\}\,,
$$
where $M^*$ denotes the complex conjugate transpose of $M$, 
and $J$ denotes the skew-symmetric matrix
$$
J = \begin{pmatrix}
     0 & 1 & 0 & 0 & 0 & 0 \\
    -1 & 0 & 0 & 0 & 0 & 0 \\
     0 & 0 & 0 & 1 & 0 & 0 \\
     0 & 0 & -1 & 0 & 0 & 0 \\
     0 & 0 & 0 & 0 & 0 & 1 \\
     0 & 0 & 0 & 0 & -1 & 0 \\
 \end{pmatrix} \,.
$$ 
Roughly, the Sato-Tate group attached to $C$ is defined  to be a compact subgroup 
$\operatorname{ST}(C) \subseteq \operatorname{USp}(6,\C)$ such that the characteristic polynomials of the matrices in
$\operatorname{ST}(C)$ {\it fit well} with
the normalized local factors
$L_p^{\operatorname{ST}}(C,T)$,  in the sense that  the normalized local factors $L_p^{\operatorname{ST}}(C,T)$, as $p$ varies, are equidistributed with respect to the Haar measure 
 of
$\operatorname{ST}(C)$ projected on the set of its conjugacy classes. 

In analogy with Galois theory, the presence of some extra structure on $C$ gives rise to proper subgroups of the symplectic group; moreover, the distribution of $L_p^{\operatorname{ST}}(C,T)$ can be viewed as a generalization of the classical Chebotarev distribution. 
Serre \cite{Ser12} proposes a vast generalization of the Sato-Tate conjecture (born for elliptic curves) giving a precise recipe for 
$\operatorname{ST}(C)$. 
In this section, we calculate the Sato-Tate group
$\operatorname{ST}(C)$ for our 
Picard curve $C$.

\begin{prop} \label{STgroup} Up to conjugation in 
$\operatorname{USp}(6,\C)$, the Sato-Tate group of $C$ is 
$$
\operatorname{ST}(C) = \left\langle \begin{pmatrix}
u_1&&&&&\\
&\bar{u}_1&&&&\\
&&u_2&&&\\
&&&\bar{u}_2&&\\
&&&&u_3&\\
&&&&&\bar{u}_3\\
\end{pmatrix}, 
\begin{pmatrix}
0 & 0 & 1 & 0 & 0 & 0 \\
0 & 0 & 0 & 1 & 0 & 0 \\
0 & 0 & 0 & 0 & 1 & 0 \\
0 & 0 & 0 & 0 & 0 & 1 \\
0 & -1 & 0 & 0 & 0 & 0 \\
1 & 0 & 0 & 0 & 0 & 0
\end{pmatrix}, |u_i|=1 \right\rangle
$$
In particular, there is an isomorphism $\operatorname{ST}(C) \simeq {U}(1)^3 \rtimes
(\Z/9\Z)^*$.
\end{prop}

\begin{proof} The recipe of Serre in \cite{Ser12} is as follows. Fix an auxiliary prime $\ell$ of good reduction (say $\ell>3$), and fix an embedding
$\iota \colon \Q_\ell \hookrightarrow \C$. Let
$$
\rho_\ell \colon \GalQ \to \GL( V_\ell(\Jac(C)))\simeq \GL(6,\Q_\ell) 
$$
be the $\ell$-adic Galois representation attached to the $\ell$-adic Tate module of the Jacobian variety of $C$. Denote by $G$ the Zariski closure of the image $\rho_\ell(\GalQ)$, and let
$G_1$ be the Zariski closure of $G\cap \operatorname{Sp}_6(\Q_\ell)$,
where $\operatorname{Sp}_6$ denotes the symplectic group.
By definition, the Sato-Tate group $\operatorname{ST}(C)$ is a maximal compact subgroup of 
$G_1\otimes_\iota \C$. In general, one hopes that this construction does not depend on $\ell$ and $\iota$, and this is the case for our Picard curve $C$. Indeed, since the CM-type of $\Jac(C)$ is 
non-degenerate then the twisted Lefschetz group $\operatorname{TL}(C)$ 
satisfies $G_1 = \operatorname{TL}(C)\otimes \Q_\ell$ for all primes $\ell$ (see \cite[Lemma 3.5]{FGL14}). Recall that the twisted Lefschetz 
group is defined as
$$
\operatorname{TL}(C) = \bigcup_{\tau\in \GalQ} \operatorname{L}(C)(\tau)\,,
$$
where $\operatorname{L}(C)(\tau) = \{ \gamma \in \operatorname{Sp}_6(\Q) \colon 
\gamma \alpha \gamma^{-1} = \tau(\alpha) \text{ for all } \alpha \in \End(\Jac(C)_{\Qbar})\otimes \Q\}$, where $\Jac(C)_\Qbar$ denotes the base change to $\Qbar$. Here, $\alpha$ is seen as an endomorphism of $H_1(\Jac(C)_{\C} ,\Q) $. 
The reason why the CM-type of $\Jac(C)$ is non-degenerate is due
to the fact that $\Phi^*$ is simple and $\dim \Jac(C)=3$ (see \cite{K,R});
alternatively, one checks that the $\Z$-linear map:
$$
\Z[\Gal(K/\Q)] \to \Z[\Gal(K/\Q) ], \qquad \sigma_a
\mapsto  \sum_{\sigma_b\in \Phi} \sigma_{b}^{-1}\sigma_a
$$
has maximal rank $1+\dim(\Jac(C))=4$. Then, by combining \cite{BGK03} and \cite[Thm.2.16(a)]{FKRS12}, it follows that the connected component
of the identity $\operatorname{TL}(C)^0$ satisfies
$$G_1^0 = \operatorname{TL}(C)^0 \otimes \Q_\ell 
= \left\{ \operatorname{diag}(x_1,y_1,x_2,y_2,x_3,y_3) \mid x_i,y_i\in \Q_\ell^*\,,
x_i y_i = 1 \right\}\,.
$$
Thus, the connected component
of the Sato-Tate group for $C$ is equal to
$$\operatorname{ST}(C)^0 = 
\left\{ 
\operatorname{diag}(u_1,\overline{u}_1,u_2,
\overline{u}_2,u_3,\overline{u}_3) 
\colon u_i \in U(1)
\right\} \simeq U(1)^3\,.
$$
According to \cite[Prop. 2.17]{FKRS12},
it also follows that the group
of components of $\operatorname{ST}(C)$
is isomorphic to $\Gal(K/\Q)$. We claim that
$
\operatorname{ST}(C) = 
\operatorname{ST}(C)^0 \rtimes \langle \gamma \rangle 
$,
where 
$$\gamma =
\begin{pmatrix}
0 & 0 & 1 & 0 & 0 & 0 \\
0 & 0 & 0 & 1 & 0 & 0 \\
0 & 0 & 0 & 0 & 1 & 0 \\
0 & 0 & 0 & 0 & 0 & 1 \\
0 & -1 & 0 & 0 & 0 & 0 \\
1 & 0 & 0 & 0 & 0 & 0
\end{pmatrix}
 \,.
$$
To this end, we consider the automorphism of the Picard curve $C$ determined by
$\alpha(x,y)=(\zeta^6x,\zeta^2\,y)$. We still denote by $\alpha$ the induced endomorphism of
$\Jac(C)$. Under the basis of regular differentials of $\Omega^1(C)$:

$$
\omega_1 = \frac{dx}{y^2} \,,\quad
\omega_2 = \frac{dx}{y}   \,,\quad
\omega_3 = \frac{xdx}{y^2} \,, 
$$
the action induced is given by $\alpha^*(\omega_1)= \zeta^2 \omega_1$, 
$\alpha^*(\omega_2)= \zeta^4 \omega_2$, 
$\alpha^*(\omega_3)= \zeta^8 \omega_3$.
By taking the symplectic basis of 
$H_1(\Jac(C)_\C,\C)$ corresponding to the
above basis (with respect to the skew-symmetric matrix $J$), we get the matrix 
$$\alpha =
\begin{pmatrix}
     \zeta^2 & 0 & 0 & 0 & 0 & 0 \\
     0 & \overline{\zeta}^2 & 0 & 0 & 0 & 0 \\
     0 & 0 & \zeta^4 & 0 & 0 & 0 \\
     0 & 0 & 0 & \overline{\zeta}^4 & 0 & 0 \\
     0 & 0 & 0 & 0 & \zeta^8 & 0 \\
     0 & 0 & 0 & 0 & 0 & \overline{\zeta}^8 \\
 \end{pmatrix} \,.
$$
One checks that the matrix $\gamma$
satisfies 
$$
\gamma \alpha \gamma^{-1} = {\,}^{\sigma_2} \alpha \,,
$$
which implies that $\gamma\in \operatorname{TL}(\sigma_2)$. Hence, $\gamma$ belongs to $\operatorname{ST}(C)$; finally, a short computations shows that
$\gamma^6=-\operatorname{Id} \in \operatorname{ST}(C)^0$, but
$\gamma^i$ is not in $\operatorname{ST}(C)^0$
for $1\leq i < 6$.

\end{proof}

\vskip 0.3truecm

\begin{rem}\label{Poly-shape}
For future use, we compute the shape of the 
characteristic polynomials in each component of the Sato-Tate group.
To this end, we take a random matrix $\operatorname{diag}(u_1,\overline{u}_1,u_2,\overline{u}_2,u_3,\overline{u}_3)$ in the connected component 
$\operatorname{ST}^0(C)$, and we get:

$$
\begin{array}{l@{\, :\quad} l}
\operatorname{ST}^0(C)\cdot \operatorname{Id} & \prod_{i=1}^3 (T-u_i)(T-\overline{u}_i) \\[3pt]
\on{ST}^0(C)\cdot \gamma & T^6+1 \\[3pt]
\on{ST}^0(C)\cdot \gamma^2 &  T^{6} + \left(u_{1} \overline{u}_2 u_{3} + \overline{u}_1 u_{2}
\overline{u}_3 \right) T^{3} + 1 \\[3pt]
\on{ST}^0(C)\cdot \gamma^3 &  {\left(T^{2} + 1\right)}^{3} \\[3pt]
\on{ST}^0(C)\cdot \gamma^4 & T^{6} - \left(u_{1} \overline{u}_2 u_{3} + \overline{u}_1 u_{2}
\overline{u}_3 \right) T^{3} + 1 \\[3pt]
\on{ST}^0(C)\cdot \gamma^5 &  T^6+1 \,.\\[3pt]
\end{array}
$$
\end{rem}

\begin{rem}\label{Remark}
As a consequence of \cite[Prop. 2.17]{FKRS12}, we also obtain that, 
for every subextension
$K/K'/\Q$, one has $\on{ST}(C_{K'}) = \on{ST}(C)^0 \rtimes \langle \gamma^{[K':\Q]} \rangle$, where $C_{K'}$ denotes the base change
$C\times_\Q K'$.
\end{rem}

\section{Sato-Tate distribution}

A general strategy to prove the expected distribution is due to Serre \cite{Ser68}. For every non-trivial irreducible representation 
$\phi\colon \operatorname{ST}(C) \to \GL_m(\C)$, one needs to consider the $L$-function
$$
L(\phi, s) = \prod_{p\neq 3} {\det(1-\phi(x_p) p^{-s} ) }^{-1}\,,
$$
where $x_p = \frac{1}{\sqrt{p}} 
\rho_{\ell}(\Frob_p) \in 
\operatorname{ST}(C)$, and then show that $L(\phi, s)$ is invertible,
in the sense that  it has meromorphic continuation to $\operatorname{Re}(s)\geq 1$ and it holds
$$
L(\phi, 1) \neq 0\,.
$$

\begin{prop} \label{prop3.1} The Picard curve 
$C\colon y^3 = x^4 -x$
satisfies the generalized Sato-Tate conjecture.
More explicitly, the sequence
$$
\left\{
\left( \frac{{\,}^{\sigma_2}\psi(\gp)}{\sqrt{N(\gp)} } 
, \frac{{\,}^{\sigma_4}\psi(\gp)}{\sqrt{N(\gp)} } , 
\frac{{\,}^{\sigma_8}\psi(\gp)}{\sqrt{N(\gp)} },\, p
\right) \right\}_{p\neq 3} \subseteq 
U(1)^3 \rtimes (\Z/9\Z)^* \simeq \on{ST}(C)\,,
$$
where $\gp$ is any prime ideal of the factorization of $p\mathcal{O}$, is equidistributed over $U(1)^3 \rtimes (\Z/9\Z)^*$
with respect to the Haar measure.
\end{prop}

\nn
\begin{proof}
The irreducible representations of $\operatorname{ST}(C)\simeq U(1)^3 \rtimes (\Z/9\Z)^*$
can be described as follows (see \cite[\S 8.2]{Ser77}). For every triple $\underline{b}=(b_1,b_2,b_3)$ in $\Z^3$, we consider the
irreducible character of $U(1)^3$ given by
$$
\phi_{\underline{b}} \colon U(1)^3 \to \C^* \,,\quad
\phi_{\underline{b}}(u_1,u_2,u_3) = \prod_{i=1}^3 u_i^{b_i}\,, 
$$
and let 
$$
H_{\underline{b}}  =
\{ h \in (\Z/9\Z)^* \colon
\phi_{\underline{b}}(u_1,u_2,u_3)  = 
\phi_{\underline{b}}({\,}^h (u_1,u_2,u_3)) \}\,. 
$$
The action of $(\Z/9\Z)^*$ on $U(1)^3$ is given by  conjugation through powers of the matrix $\gamma$; more precisely, for the generator $g=2$ of $(\Z/9\Z)^*$ we have ${\,}^g (u_1,u_2,u_3) = (u_2,u_3,\overline{u}_1)$ since
$$
\gamma 
\begin{pmatrix}
u_1&&&&&\\
&\bar{u}_1&&&&\\
&&u_2&&&\\
&&&\bar{u}_2&&\\
&&&&u_3&\\
&&&&&\bar{u}_3\\
\end{pmatrix} \gamma^{-1}
=
\begin{pmatrix}
u_2&&&&&\\
&\bar{u}_2&&&&\\
&&u_3&&&\\
&&&\bar{u}_3&&\\
&&&& \bar{u}_1 &\\
&&&&& u_1\\
\end{pmatrix}
\,.
$$
An easy computation shows that $H_{\underline{b}}=
\langle 2 \rangle$ or $\langle 2^3 \rangle$
if and only if $\underline{b}=(0,0,0)$, while 
$H_{\underline{b}} = \langle 2^2 \rangle$ for 
$\underline{b}=(b_1,-b_1,b_1)$ with $b_1\neq 0$, and $H_{\underline{b}}$ is trivial otherwise.
Then, one has that 
$$
\phi_{\underline{b}}(u_1,u_2,u_3,h) = \prod_{i=1}^3 u_i^{b_i}
$$
is a character of 
$H:=U(1)^3 \rtimes H_{\underline{b}}$.
By \cite[Prop. 25]{Ser77} every irreducible representation of $G:=U(1)^3 \rtimes (\Z/9\Z)^*$
is of the form 
$\theta := \on{Ind}_H^G ( \phi_{\underline{b}}\otimes \chi )$, where $\chi$ is a character of $H_{\underline{b}}$ that may be viewed as a character of $H$ by composing with the projection $H \to  H_{\underline{b}}$.

Let $\theta=\on{Ind}_H^G ( \phi_{\underline{b}}\otimes \chi )$ be an irreducible representation of 
$U(1)^3 \rtimes (\Z/9\Z)^*$ as above. If we denote
the sequence by
$$
x_{p} = 
\left( \frac{{\,}^{\sigma_2}\psi(\gp)}{\sqrt{N(\gp)} } 
, \frac{{\,}^{\sigma_4}\psi(\gp)}{\sqrt{N(\gp)} } , 
\frac{{\,}^{\sigma_8}\psi(\gp)}{\sqrt{N(\gp)} },\, p
\right) \in U(1)^3 \rtimes (\Z/9\Z)^*
$$
where $\gp$ is any prime ideal of the factorization of $p\mathcal{O}$, our claim is equivalent to show that 
the corresponding $L$-function 
$$
L(\theta, s)  = 
\prod_{p \neq 3} 
(1-\det (\theta(x_{p})) p^{-s})^{-1}
$$ is invertible provided that $(b_1,b_2,b_3) 
\neq (0,0,0) $. Assume first that 
$H_{\underline{b}}$ is trivial. 
Then, also $\chi$ is trivial and one has
$$
L(\theta,s) = L(\phi_{\underline{b}},s) = 
\prod_{p\neq 3} \left( 
1- \frac{{\,}^{\sigma_2}\psi(\gp)^{b_1}
{\,}^{\sigma_4}\psi(\gp)^{b_2}
{\,}^{\sigma_8}\psi(\gp)^{b_3}}
{\sqrt{\on{N}(\gp)}^{b_1+b_2+b_3}} \right)\,.
$$
This can be seen as the $L$-function of the unitarized Gr\"ossencharakter
$$
\Psi := \frac{
{\,}^{\sigma_2}\psi(\cdot)^{b_1}
{\,}^{\sigma_4}\psi(\cdot)^{b_2}
{\,}^{\sigma_8}\psi(\cdot)^{b_3} }
{{\on{N}(\cdot)}^{(b_1+b_2+b_3)/2}}
$$
Under our assumption $(b_1,b_2,b_3) 
\neq (0,0,0) $ and by using the factorization of
$\psi(\gp)\cO$ into prime ideals (see property (iii) in the proof of Proposition \ref{prop1.1}), an easy computation shows that $\Psi$ is non-trivial. Hecke showed \cite{He20} that the $L$-function of a non-trivial unitarized 
Gr\"ossencharkter is holomorphic and non-vanishing for $\on{Re}(s)\geq 1$. In the remaining case, that is for $H_{\underline{b}}$ of order $3$, one gets $L(\Psi,s)=L(\theta,s)^3$ and the claim also follows by the same argument.
\end{proof}

\vskip 0.3truecm

\section{The moment sequences}

In this section we will compute the moment sequences in two independent ways, one (exact) from the Sato-Tate group and the other one (numerically) by computing the local factors of our curve up to some bound.

\medskip

Let $\mu$ be a positive measure on $I = [-d,d]$. Then, on the one hand, for every integer $n\geq 0$, the $n$th moment $M_n[\mu]$ is by definition 
$\mu(\varphi_n)$, where $\varphi_n$ is the function $z \mapsto 
z^n$.
That is, we have
\begin{equation*}
M_n[\mu] = \int_{I} \, z^n\mu(z)
\end{equation*}
The measure $\mu$ is \emph{uniquely} determined by its moment sequence $M_n[\mu]$. 

On the other hand, if a sequence $\{a(p)\}_p$ is $\mu$-equidistributed, then the following equality holds:
$$ M_n[\mu] = \lim_{x\to\infty} \frac{1}{\pi(x)}\sum_{p\leq x} a_p^n \,.$$

\medskip

From now on, we shall denote by $a_1(p)$, $a_2(p)$, $a_3(p)$ the {\sl higher} traces according to
$$
L_p^{\operatorname{ST}}(C,T)=
1+ a_1(p)T+ a_2(p) T^2+ a_3(p) T^3 +a_2(p)T^4
+a_1(p)T^5+T^6 \,.
$$

Recall that due to the Weil's conjectures, we know that
$$
a_1(p) \in [-6,6]\,,\quad
a_2(p) \in [-15,15]\,,\quad
a_3(p) \in [-20,20]\,.
$$

\subsection{The distribution of $ST(C)$}

For each $i$ in $\{ 1,2,3 \}$, let $\mu_i$ denote the projection on the
interval $I_i=[-\binom{6}{i},+\binom{6}{i} ]$ obtained from the Haar measure of the Sato-Tate group $\operatorname{ST(C)}\simeq {U}(1)^3 \rtimes
(\Z/9\Z)^*$.

In general it is difficult to obtain the explicit distribution function, but because of the isomorphism stated in Proposition \ref{STgroup}, we can easily compute the moment sequence of the Sato-Tate measure.

Similarly as in  \cite{FGL14} we shall split each measure as a sum of its restrictions to each component of $\on{ST}(C)^0 \cdot \gamma^k$, where $0\leq k \leq 5$.

Therefore one has
$$
\mu_i = \frac{1}{6} \sum_{0\leq k \leq 5} {\,}^k\mu_{i} \,, \quad M_n[\mu_i] = \frac{1}{6}\sum_{0\leq k \leq 5} M_n[{\,}^k\mu_{i}]
$$
so we can compute the moments $M_n[{\,}^k\mu_{i}]$ separately for every $k$ and then get the \emph{total} moments $M_n[\mu_{i}]$. To ease notation, we shall denote  the moment sequences by
$$
M[\mu_i]:= ( M_0[\mu_i], M_1[\mu_i], M_1[\mu_i], \dots , M_n[\mu_i], \dots ) \,,
$$
and similarly for every $M[{\,}^k\mu_{i}]$.

In what follows, the characteristic polynomial of a matrix in $\on{USp}(6)$ we will be denoted by
$$
P(T)=
1+ a_1T+ a_2 T^2+ a_3 T^3 +a_2T^4+a_1T^5+T^6 \,.
$$

\noindent Case $k=1,5$: In these components, according to Remark \ref{Poly-shape} one has that
$P(T) = T^6+1$, so that $$a_1 = a_2 = a_3=~0\,.$$
Hence, 
$$M_n[{\,}^k\mu_{1}] = M_n[{\,}^k\mu_{2}] = M_n[{\,}^k\mu_{3}] = 0 \,
\text{ for all }
n\geq 1\,.$$ 

\noindent Case $k=2,4$: In these components, we have
$$P(T) = T^6 \pm (u_1\overline{u}_2u_3 + \overline{u}_1u_2\overline{u}_3)T^3 + 1 \,.$$
So that $a_1 = a_2 = 0$.
Hence, it follows that
$$M_n[{\,}^k\mu_{1}] = M_n[{\,}^k\mu_{2}] = 0 \text{ for all }
n\geq 1\,.$$
To get the distribution of the third trace, 
since $u_1$, $u_2$, and $u_3$ are independent elements of $\on{U}(1)$, the distribution of $a_3(p)$ will correspond to the distribution of $\alpha := u + \overline{u}$ for $u\in\on{U}(1)$, and hence its associated moment sequence is
$$M[{\,}^k\mu_{3}] = (1,0,2,0,6,0,20,0,\dots)\,.$$

\noindent Case $k=3$: In this case, one has
$P(T)= (1 + T^2)^3$, so that we have
$a_1 = a_3 = 0$, while $a_2 = 3$.
Hence, we obtain
$$M_n[{\,}^3\mu_{1}] = M_n[{\,}^3\mu_{3}] = 0\,,\ 
  M_n[{\,}^3\mu_{2}] = 3^n \ \text{ for all }
n\geq 1\,.$$

\noindent Case $k=0$: In this case one has that 
$P(T) = \prod_{i=1}^3 (T-u_i)(T-\overline{u}_i)$. If we develop this expression we get the following coefficients, where as above $\alpha_i$ stands for the sum of $u_i$ and its complex conjugate:
\begin{equation*}\label{aidealfa}
\begin{array}{l@{\,=\,}l}
a_1 & \alpha_1+\alpha_2+\alpha_3\,,\\[3pt]
a_2 & 3 + \alpha_1\alpha_2 + \alpha_2\alpha_3 + \alpha_1\alpha_3\,,\\[3pt]
a_3 & 2\,\alpha_1 + 2\,\alpha_2 + 2\,\alpha_3 + \alpha_1\alpha_2\alpha_3.
\end{array}
\end{equation*}

To get the sequences we proceed as follows. Recall that 
if $X$ and $Y$ denote independent random variables, then $M_n[X] = E(X^n)$, $E(X+Y) = E(X)+E(Y)$, and $E(XY) = E(X)E(Y)$. Hence, one has
\begin{align*}
M_n[X+Y] = E((X+Y)^n) =& E\left(\sum_{k=0}^n\binom{n}{k}X^kY^{n-k}\right)\\
 =& \sum_{k=0}^n\binom{n}{k}E(X^k)E(Y^{n-k})\\ =& \sum_{k=0}^n\binom{n}{k}M_k[X]M_{n-k}[Y]\,.
\end{align*}

Since we know that $M[\alpha]:=M[\alpha_i] = (1,0,2,0,6,0,20,0,\dots)$ for $i=1,2,3$, one gets:
\begin{align*}
M_n[{\,}^0 \mu_{1}] =& \sum_{a+b+c=n}\binom{n}{a,b,c}M_a[\alpha]M_b[\alpha]M_c[\alpha],\\
M_n[{\,}^0\mu_{2}] =& \sum_{a+b+c+d=n}\binom{n}{a,b,c,d}3^aM_{b+d}[\alpha]M_{b+c}[\alpha]M_{c+d}[\alpha]\,,\\
M_n[{\,}^0\mu_{3}] =& \sum_{a+b+c+d=n}\binom{n}{a,b,c,d}2^{a+b+c}M_{a+d}[\alpha]M_{b+d}[\alpha]M_{c+d}[\alpha].
\end{align*}

Therefore we obtain the sequences:
\begin{align*}
M[{\,}^0\mu_{1}] =& (1,0,6,0,90,0,1860,\dots),\\
M[{\,}^0\mu_{2}] =& (1,3,21,183,1845,\dots),\\
M[{\,}^0\mu_{3}] =& (1,0,32,0,4920,0,1109120,\dots).
\end{align*}

We can summarize the above results in the following proposition.

\begin{prop} With the above notations, the first moments of the
measures of ${\,}^k\mu_i$ and $\mu_i$ are as follows:

\begin{itemize}
\item[(i)] {The moments of the first trace are:} 
$$M[{\,}^k\mu_1] = \begin{cases}
(1,0,0,\dots) &\text{if } k = 1,\dots,5\,;\\
(1,0,6,0,90,0,1860,\dots) & \text{if } k = 0\,.
\end{cases}$$

Hence, $M[\mu_1] = (1,0,1,0,15,0,310,\dots)$. 

\item[(ii)]{The moments of the second trace are:} 
$$M[{\,}^k\mu_2] = \begin{cases}
(1,0,0,\dots) &\text{if } k = 1,2,4,5\,;\\
(1,3,9,27,\dots) &\text{if } k = 3\,;\\
(1,3,21,183,1845,\dots) &\text{if } k = 0\,.
\end{cases}$$

Hence, $M[\mu_2] = (1,1,5,35,321,\dots)$.

\item[(iii)]{The moments of the third trace are:} 
$$M[{\,}^k\mu_3] = \begin{cases}
(1,0,0,\dots) &\text{if } k = 1,3,5\,;\\
(1,0,2,0,6,0,20,0,\dots) &\text{if } k = 2,4\,;\\
(1,0,32,0,4920,0,1109120,\dots) &\text{if } k = 0\,.
\end{cases}$$

Hence, $M[\mu_3] = (1,0,6,0,822,0,184860,0\dots)$.
\end{itemize}

\end{prop}

\subsection{The numerical sequences for $C$\label{numeric}}
Once we have computed the theoretical moment sequences from 
the Sato-Tate group $\on{ST}(C)$, we wish to compute for every prime (up to some bound) its associated normalized local factor $L_p^{\on{ST}}(A,T)$ to get the corresponding traces $a_1(p), \, a_2(p)$ and $a_3(p)$ and do the experimental equidistribution matching.

The Gr\"ossencharakter $\psi$ attached to the Picard curve $C$ permits us to perform this numerical experimentation within a reasonable time, in this case $p\leq 2^{26}$ (about two hours of a standard laptop). We display the data obtained:

\begin{table}[h]
$$\begin{array}{rcrrcrrcrr}
\hline
&&\multicolumn{2}{c}{a_1}&&\multicolumn{2}{c}{a_2}&&\multicolumn{2}{c}{a_3}\\
n&&M_n[\mu_1]&M_n[\mu_1]_{\leq 2^{26}}&&M_n[\mu_2]&M_n[\mu_2]_{\leq 2^{26}}&&M_n[\mu_3]&M_n[\mu_3]_{\leq 2^{26}}\\
\hline
0&&1&1&&1&1&&1&1\\
1&&0&-0.000&&1&0.999&&0&-0.000\\
2&&1&0.998&&5&4.991&&6&5.984\\
3&&0&-0.005&&35&34.868&&0&-0.147\\
4&&15&14.946&&321&319.058&&822&815.937\\
5&&0&-0.151&&&&&&\\
6&&310&308.160&&&&&&\\
\hline
\end{array}$$
\caption{Numerical moment sequences computed for $p$ up to $2^{26}$.}
\end{table}

We include graphics to display the histograms (for primes up to $p\leq 2^{26}$)
showing the nondiscret components of the three distributions $\mu_i$. 

\begin{figure}[H]
\centering
\includegraphics[scale=.75]{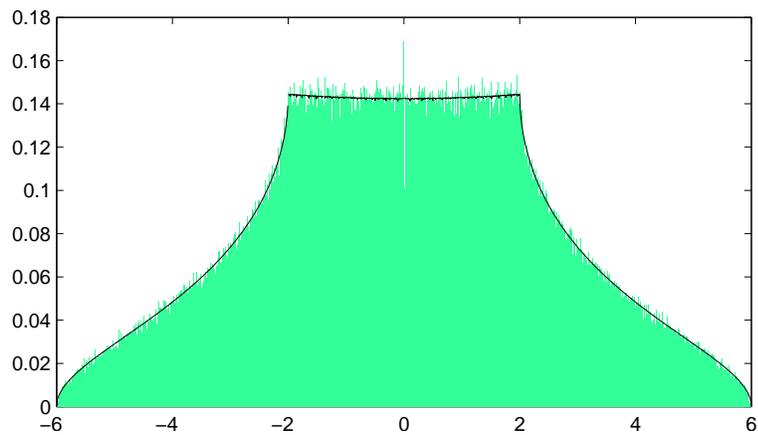}
\caption{Histogram of the first trace for primes $p\equiv 1 \, (\!\bmod 9)$.}
\label{fig:digraph1}
\end{figure}

\begin{figure}[H]
\centering
\includegraphics[scale=.75]{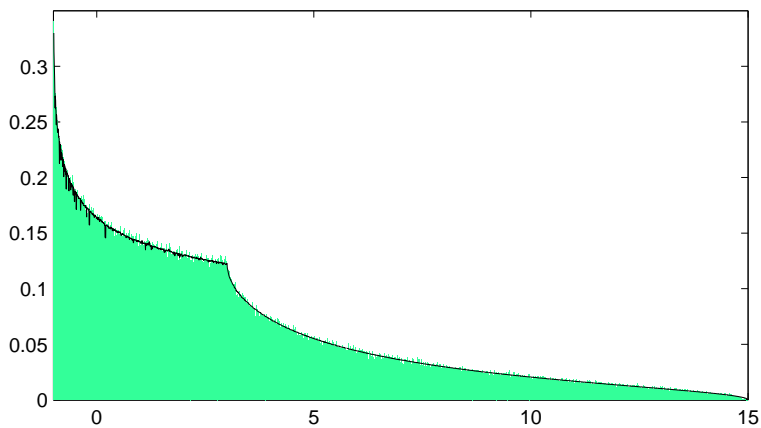}
\caption{Histogram of the second trace for primes $p\equiv 1 \, (\!\bmod 9)$.}
\label{fig:digraph2}
\end{figure}

\begin{figure}[H]
\centering
\includegraphics[scale=.75]{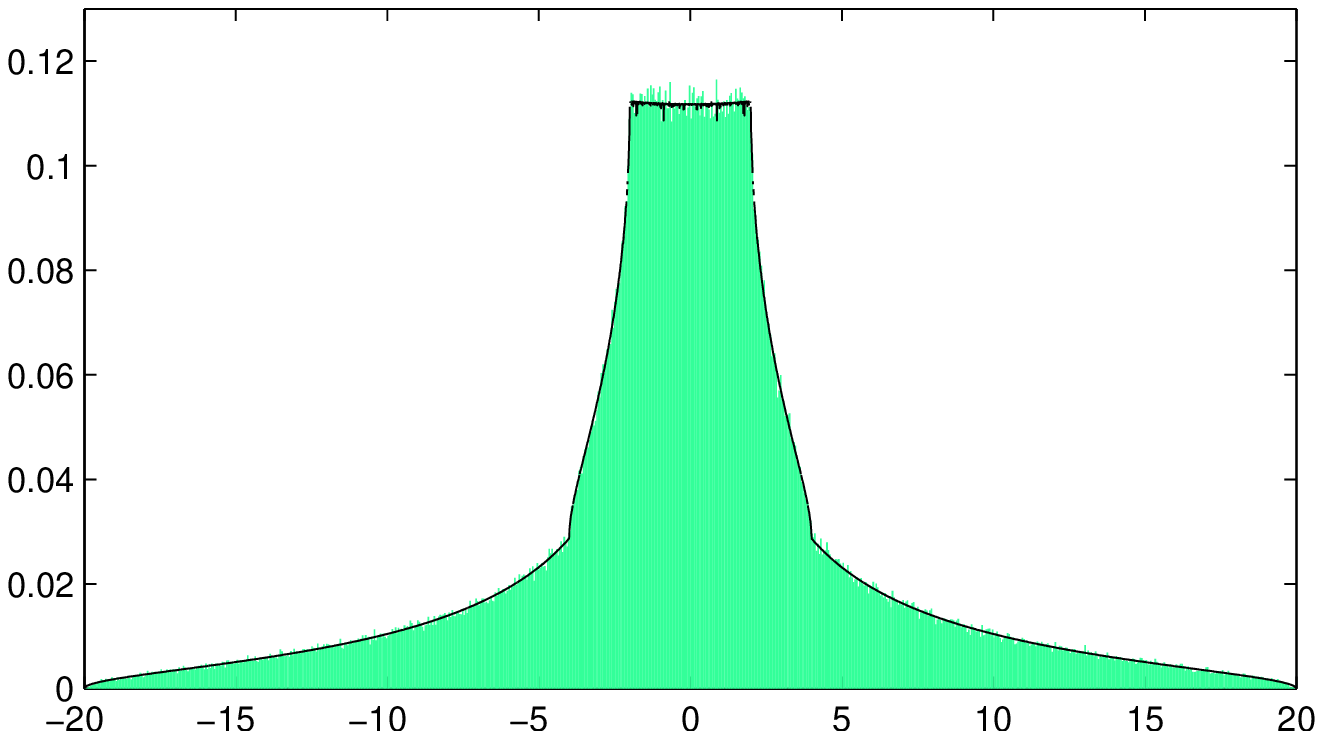}
\caption{Histogram of the third trace for primes $p\equiv 1 \, (\!\bmod 9)$.}
\label{fig:digraph3}
\end{figure}

\begin{figure}[H]
\centering
\includegraphics[scale=.75]{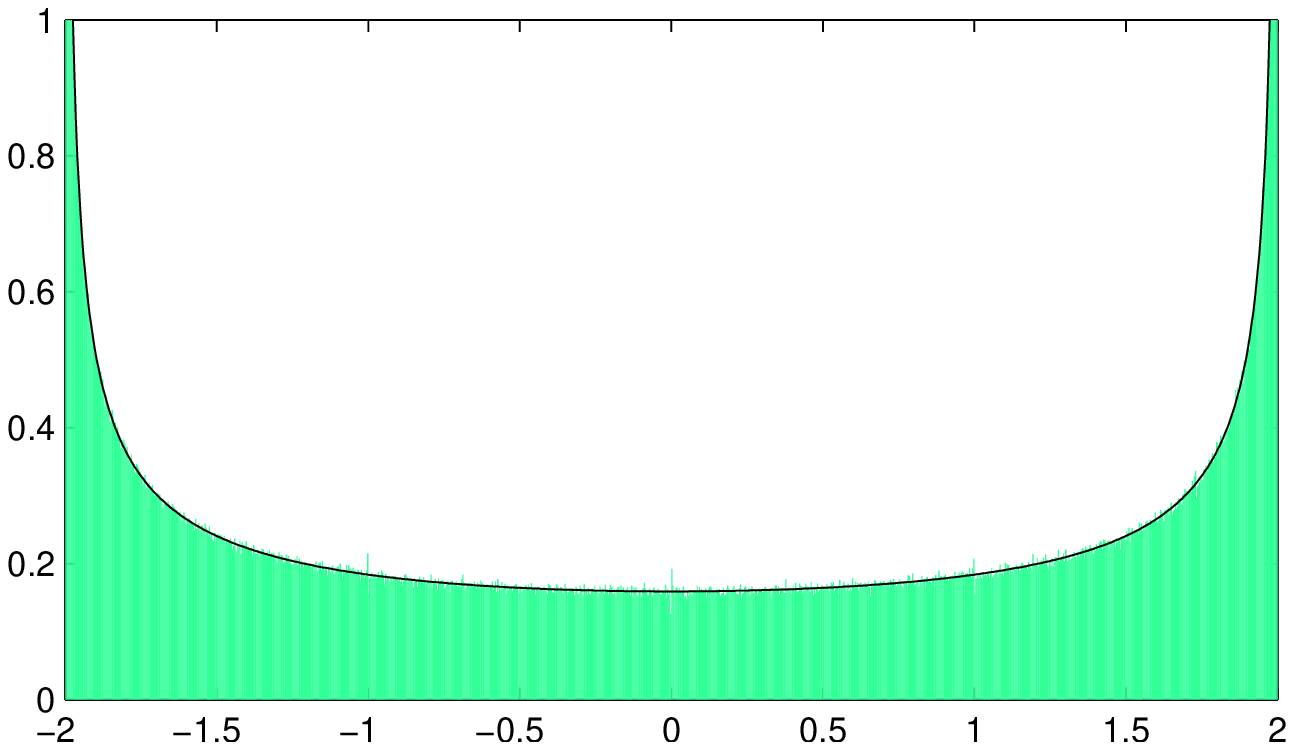}
\caption{Histogram of the third trace for primes $p\equiv 4,7 \, (\!\bmod 9)$.}
\label{fig:digraph4}
\end{figure}

\bigskip

\nocite{*}
\bibliographystyle{amsalpha}
\bibliography{ref}

\newpage

\section*{Appendix (by F. Fité)}
\setcounter{section}{1}
\setcounter{teo}{0}
\setcounter{prop}{0}
\renewcommand{\thesection}{\Alph{section}}

We keep the notation of the article. Let $C\colon y^ 3=x^ 4-x$. Let $K$ denote the cyclotomic field $\Q(\zeta)$, where $\zeta$ is a 9th root of unity. For every prime $\p$ of $K$ coprime to 3, consider the character $\chi_\p\colon \F_\p^*\rightarrow K^*$ such that $\chi_\p(x)$ is the only 9th root of unity satisfying
$$
\chi_\p(x)\equiv x^{(N(\p)-1)/9}\pmod \p\,.
$$
For $a,b\in \Z/9\Z$, define
$$
J_{(a,b)}(\p):=\sum_{x\in\F_\p}\chi_\p^a(x)\chi^b_\p(1-x)\,.
$$
 
\begin{prop} The number of points of $ C$ defined over the finite field $\F_\p$ is
$$
| C(\F_\p)|=
\begin{cases}
1 + N(\p) & \text{if $N(\p) \not\equiv 1 \pmod 9$,}\qquad \text{(i)}\\
1+N(\p)+ \Tr_{K/\Q}(J_{(6,1)}(\p)) &  \text{if $N(\p) \equiv 1 \pmod 9$.}\qquad \text{(ii)}\\
\end{cases}
$$
\end{prop}

\begin{proof}
Case $(i)$ is considered in Proposition 1 and Proposition 2 of \cite{HN02}. We now show case (ii), by giving an alternative and shorter proof of Proposition 3 of \cite{HN02}. Let  $
 C'\colon v^ 9=u(u+1)^ 6$. 
There is an isomorphism between~$ C$ and~ $ C'$ given by
$$
\phi\colon  C \rightarrow  C'\,, \qquad \phi(x,y)=\left(-\frac{1}{x^3},-\frac{y^ 2}{x^3} \right)\,.
$$
One easily sees that the inverse of $\phi$ is given by
$$
\phi^{-1}\colon  C' \rightarrow  C\,, \qquad \phi^{-1}(u,v)=\left(-\frac{(u+1)^2}{v^3},-\frac{(u+1)^3}{v^4} \right)\,.
$$
Note that if $N(\p) \not\equiv 1 \pmod 3$, then exponentiation by 9 is an isomorphism of $\F_\p$. Thus~$ C'$ has~$N(\p)$ affine points plus one point at infinity. Assume now that $N(\p) \equiv 1 \pmod 9$. By \cite[Prop. 8.1.5]{IR90}, we have that
\begin{equation*}
\begin{aligned}
| C'(\F_\p)| &=1+\sum_{u\in\F_\p}\sum_{a\in\Z/9\Z}\chi_\p^a(u)\chi_\p^{6a}(u+1)\\
&=1+N(\p)+\sum_{a\in(\Z/9\Z)^*}\sum_{u\in\F_\p}\chi_\p^a(u)\chi_\p^{6a}(u+1)\,,
\end{aligned}
\end{equation*}
where for the second equality we have used \cite[Thm. 1 (b), p. 93]{IR90}. But writing $x=u+1$, we obtain
$$
\sum_{u\in\F_\p}\chi_\p(u)\chi_\p^{6}(u+1)=\chi_\p(-1)\sum_{x\in\F_\p}\chi_\p^{6}(x)\chi_\p(1-x)\,.
$$
Case $(ii)$ of the proposition is a consequence of the equality $\chi_\p(-1)=1$ (this follows form the fact that the order of $\chi_\p$ is odd).
\end{proof}

To show that our result agrees with Proposition 3 of \cite{HN02} it remains to show that
$
\Tr_{K/\Q}(J_{(6,1)}(\p))=\Tr_{K/\Q}(J_{(3,1)}(\p))\,.
$
Indeed, by \cite[Thm. 2.1.5]{BEW98}, one has
$$
J_{(6,1)}(\p)=J_{(2,1)}(\p)\,,\qquad J_{(3,1)}(\p)=J_{(5,1)}(\p)\,.
$$
Since $5\cdot 2\equiv 1 \pmod{9}$, we deduce that 
$$
\Tr_{K/\Q}(J_{(2,1)}(\p))=\Tr_{K/\Q}(J_{(5,1)}(\p))\,.
$$

Finally, note that $J(\p)=-J_{(3,1)}(\p)$ in the notation of the article.

\end{document}